\documentclass[12pt]{amsart}
\usepackage{amsmath,amssymb,amscd,xy}

\newtheorem{propo}{Proposition}[section]
\newtheorem{corol}[propo]{Corollary}
\newtheorem{theor}[propo]{Theorem}
\newtheorem{lemma}[propo]{Lemma}

\theoremstyle{definition}
\newtheorem{defin}[propo]{Definition}
\newtheorem{examp}[propo]{Example}

\theoremstyle{remark}
\newtheorem{remar}[propo]{Remark}

\numberwithin{equation}{section}

\newcommand{\al }{\alpha }
\newcommand{\cEs }{\mathcal{F}}
\newcommand{\cG }{\mathcal{G}}
\newcommand{\cA }{\mathcal{A}}
\newcommand{\cT }{\mathcal{T}}
\newcommand{\cAp }{\mathcal{A}^+}

\newcommand{\cC }{\mathcal{C}}

\newcommand{\Cm }{C}
\newcommand{\cm }{c}
\newcommand{\DG }{\mathbb{D}}

\newcommand{\id }{\mathrm{id}}
\newcommand{\NN }{\mathbb{N}}

\newcommand{\QQ }{\mathbb{Q}}
\newcommand{\ZZ }{\mathbb{Z}}
\newcommand{\rfl }{\rho }
\newcommand{\Ob }{\mathrm{Ob}}

\newcommand{\re }{^\mathrm{re}}
\newcommand{\rsC }{\mathcal{R}}
\newcommand{\rsCre }{\mathcal{R}^{re}}
\newcommand{\s }{\sigma }

\newcommand{\Wg }{\mathcal{W}}
\newcommand{\detz}{\mbox{\rm det}_z}
\DeclareMathOperator{\Aut}{Aut}
\DeclareMathOperator{\End}{End}
\DeclareMathOperator{\Hom}{Hom}
\DeclareMathOperator{\Prod}{Prod}
\DeclareMathOperator{\Gr}{Gr}
\DeclareMathOperator{\Mod}{mod}

\xyoption{all}

\title[Reflection groupoids and Cluster algebras]
{Reflection groupoids of rank two and Cluster algebras of type $A$}

\author{M.~Cuntz}
\address{Michael Cuntz,
Fachbereich Mathematik,
Universit\"at Kaiserslau\-tern,
Postfach 3049,
D-67653 Kaiserslautern, Germany}
\email{cuntz@mathematik.uni-kl.de}

\author{I.~Heckenberger}
\thanks{I.H. is
supported by the German Research Foundation (DFG) via a Heisenberg
fellowship}
\address{Istv\'an Heckenberger, Philipps-Universit\"at Marburg,
Fachbereich Mathematik und Informatik,
Hans-Meerwein-Stra\ss e,
D-35032 Marburg, Germany}
\email{heckenberger@mathematik.uni-marburg.de}

\begin{document}

\begin{abstract}
We extend the classification of finite Weyl groupoids of rank two.
Then we generalize these Weyl groupoids to
`reflection groupoids' by admitting non-integral entries of the Cartan
matrices. This leads to the unexpected observation that
the spectrum of the cluster algebra of type $A_{n-3}$ completely describes
the set of finite reflection groupoids of rank two with $2n$ objects.
\end{abstract}

\maketitle

\section{Introduction}

In the last years, the classification of pointed Hopf algebras
has grown to a very fruitful subject. Amongst others, the Weyl groupoid
was invented \cite{p-H06} \cite{p-HY-08},
a structure which plays a similar role
as the Weyl group for semisimple Lie algebras.
A Weyl groupoid is a groupoid which is defined using a family of generalized
Cartan matrices.
The abstract notion of the Weyl groupoid
is perhaps somewhat too general from the viewpoint
of pointed Hopf algebras, but nevertheless its system of axioms is good
enough to provide a rich general theory and to admit
a classification at least of the finite case
(see \cite{p-CH09-2}, \cite{p-CH09-3}),
thus it has proven to be very useful anyway.

Recent work on finite rank three Weyl groupoids has revealed
a connection to combinatorics, because a Weyl groupoid yields a simplicial
arrangement in the same manner a Coxeter group does.
Since the classification of simplicial arrangements in the real projective
plane is still an open problem
- Gr\"unbaum conjectures that he has a complete list \cite{p-G-09} -
it appears now to be a natural question to generalize Weyl groupoids
to the case in which the Cartan entries are not necessarily integers:
This would at least yield one more case, the Coxeter group of type $H_3$.

However, the classification algorithm for rank three in \cite{p-CH09-3} requires a
good knowledge on finite rank two Weyl groupoids.
So to find an explanation of simplicial arrangements in terms of Weyl groupoids,
it is necessary to understand at least finite reflection groupoids of rank two
first. These are certain
groupoids generated by reflections given by a vast generalization of
Cartan matrices, see Section~\ref{regr_car} for a precise definition.

In general, we are mostly interested in Weyl groupoids which admit a
finite root system in the sense of \cite[Def.\,2.2]{p-CH09}.
A Weyl groupoid is called {\em universal} if for each object $a$
the group $\End (a)$ is trivial. It is {\em irreducible} if none of the
generalized Cartan matrices are decomposable.
In \cite{p-CH09-2}, a construction is given to obtain all finite universal
Weyl groupoids of rank two which admit a finite root system.
We refine this result by
explaining the combinatorics:

\begin{theor}[Theorem~\ref{wgr2_tri}]
There is a natural bijection between the set of
isomorphism classes of connected irreducible universal Weyl
groupoids of rank two with $2n$ objects which admit a finite root system,
and the triangulations of a convex
$n$-gon by non-intersecting diagonals, up to the symmetry of the dihedral
group $\mathbb{D}_n$.
\end{theor}

Further, we explicitly give the corresponding root systems in
Proposition~\ref{rsrank2} in terms of sequences closely related
to the well-known Farey series.
We conclude in Corollary~\ref{sumof2}
that any positive root is either simple or the sum of two positive
roots (at the same object). We also describe the quotients of universal
coverings in Proposition~\ref{end_aa}.

For the classification of connected finite reflection
groupoids of rank two we proceed in the most natural way.
We view the entries of the defining `Cartan matrices'
as indeterminates, translate the axioms for finiteness into polynomials
and consider the resulting variety. Using an induction on the number of objects
with a similar rule as in \cite{p-CH09-2}, we obtain a surprising explanation
for the theory in rank two,
in which matrix mutation appears in a completely natural way.

\begin{theor}[Theorem~\ref{thm:main}] \label{thm:frg}
    Let $n\in \NN _{\ge 3}$ and let
  $\Wg $ be a connected finite reflection groupoid of rank two
  with $2n$ objects.
  The set of connected finite reflection groupoids of
rank two with the same objects and the same object change diagram as $\Wg $
is a variety isomorphic to the spectrum of the
cluster algebra of type $A_{n-3}$,
where the edges of the $n$-gon $q_1,\ldots,q_n$ are
specialized to $1$.
\end{theor}

For the definition of the object change diagram see Section~\ref{regr_car}.

This note is organized as follows. In Section \ref{regr_car}
we recall the definition of the main structures as in \cite{p-CH09},
except that the Cartan entries may come from an arbitrary ring here.
In Section \ref{class_Weyl2} we explain the combinatorics of
finite Weyl groupoids of rank two.
In the last section, we shortly describe the Grassmannian $\Gr(2,n)$
following \cite[Lecture 3]{FR07}
and exhibit the connection to the reflection groupoids of rank two.

\section{Reflection groupoids of rank two}\label{regr_car}

Let $K$ be a ring.
We first introduce and recall some definitions and notations
needed to formulate the
definition of a finite reflection groupoid of rank two.
Remark that the only reason why we restrict to the case of rank two is that
we have no canonical choice for a system of axioms in the higher rank yet.

\begin{defin} \label{de:CS}
  Let $I$ be a set with $|I|=2$. Let $A$ be a non-empty set,
  $\rfl _i : A \to A$ a map for all $i\in I$,
  and $\Cm ^a=(\cm ^a_{jk})_{j,k \in I}$ a matrix
  in $K^{I \times I}$ for all $a\in A$. The quadruple
  \[ \cC = \cC (I,A,(\rfl _i)_{i \in I}, (\Cm ^a)_{a \in A})\]
  is called a \textit{$K$-Cartan scheme} if
  \begin{enumerate}
  \item[(M1)] $\cm^a_{ii}=2$ for all $a\in A$, $i\in I$,
  \item[(C1)] $\rfl _i^2 = \id$ for all $i \in I$,
  \item[(C2)] $\cm ^a_{ij} = \cm ^{\rfl _i(a)}_{ij}$ for all $a\in A$ and
    $i,j\in I$.
  \end{enumerate}
\end{defin}

The notion {\em $K$-Cartan scheme} has its origins in
\cite{p-CH09} where Cartan schemes are defined using a family of generalized
Cartan matrices. More precisely, a Cartan scheme is a $\ZZ$-Cartan scheme
for which all matrices $C^a$, where $a\in A$, are generalized Cartan matrices
in the sense of \cite[Sect.\,1.1]{b-Kac90}.
In Definition~\ref{de:CS} the matrices $C^a$ are not necessarily
generalized Cartan matrices, even if $K=\ZZ$.
The reason for the generality in our definition is twofold.
First, if we considered additional axioms on the matrices $C^a$ then
the interpretation of Theorem~\ref{thm:main} would be more complicated
than Theorem~\ref{thm:frg}.
Second, let $K=\ZZ$ and
assume that any real root (see below) associated to any $a\in A$
is either positive or negative. Then for all $a\in A$
the matrix $C^a$
is a generalized Cartan matrix, see \cite[Lemma~2.5]{p-CH09}.

  Two $K$-Cartan schemes are termed {\em equivalent} if there exist bijections
  between their sets $I$ respectively $A$ which satisfy the natural
  compatibility conditions, see \cite[Def.~2.1]{p-CH09} for details.

  For all $i \in I$ and $a \in A$ define $\s _i^a \in
  \Aut(K^I)$ by
  \begin{align}
    \s _i^a (\al _j) = \al _j - \cm _{ij}^a \al _i \qquad
    \text{for all $j \in I$.}
    \label{eq:sia}
  \end{align}
  Then for all $i\in I$ and $a\in A$ the linear map $\s _i^a$ is a reflection
  in the sense of \cite[Ch.\,V, \S2.2]{b-BourLie4-6}.

  To $\cC$ belongs a category $\Wg (\cC )$ such that $\Ob (\Wg (\cC ))=A$ and
  the morphisms are generated by the maps $\s _i^a$ which are by definition
  in $\Hom (a,\rfl _i(a))$ with $i\in I$, $a\in A$.

  To each object $a\in A$, one can associate a set of {\it real roots}
  \[ R^a = \{ \varphi(\alpha_i) \mid i\in I,\:\: \varphi\in\Hom(b,a),\:\:b\in A \}. \]
  In the special case $K=\ZZ$ we have a set $R^a_+=R^a\cap \NN _0^I$ of
  {\it positive roots}.
  We write $\rsCre(\cC) = (\cC,(R^a)_{a\in A})$.

\begin{remar}\label{weyl_groupoid}
  Let $\cC $ be a $\ZZ$-Cartan scheme and assume that for all $a\in A$
  the matrix $C^a$ is a generalized Cartan matrix.
  Then $\cC $ is a \textit{Cartan scheme}.
  Although for the purposes of this paper we assumed that $|I|=2$,
  in the general theory \cite{p-CH09} this assumption is not necessary.

  We say that $\cC $ is {\it irreducible} if the generalized Cartan matrix
  $C^a$ is indecomposable for any $a\in A$.
  The groupoid $\Wg(\cC)$ is called the {\it Weyl groupoid of $\cC $}
  (compare \cite{p-CH09}).
  The Weyl groupoid $\Wg (\cC )$ is {\it finite}
  if the number of objects in each connected
  component of $\Wg(\cC)$ is finite and
  if $|\Hom(a,b)|<\infty$ for all $a,b\in A$.
  We say that $\cC $ is {\it connected} if $\Wg (\cC )$ is connected.
  Further, $\cC $ is {\it simply connected} if $|\End (a)|=1$ for all
  $a\in A$.

  For the existence of a root system of type $\cC$
  in the sense of \cite[Def.~2.2]{p-CH09}
  it is necessary that the real roots are
  contained in $\NN _0^I\cup -\NN _0^I$. This is not always the case even
  if the rank of $\cC $ is two: A counterexample can be found in
  the proof of \cite[Thm.\,6.1]{p-CH09}.
  By \cite[Prop.~2.12]{p-CH09}, if there is a finite root system of type $\cC $
  then all roots are real. Therefore there exists a finite root system
  of type $\cC $ if and only if $\rsC \re (\cC )$ is a finite root system of
  type $\cC $.
\end{remar}

Let
$\cC = \cC (I,A,(\rfl _i)_{i \in I}, (\Cm ^a)_{a \in A})$
be a $K$-Cartan scheme.
We will omit the source of a morphism in the notation if it is clear from the
context.
The number $|I|$ is called the {\it rank} of $\cC$. It is always two in this
article.

Let $\Gamma $ be a non directed graph,
such that the vertices of $\Gamma $ correspond to the elements of $A$.
Assume that for all $i\in I$ and $a\in A$ with $\rfl _i(a)\not=a$
there is precisely one edge between the vertices $a$ and $\rfl _i(a)$
with label $i$, and all edges of $\Gamma $ are given in this way.
The graph $\Gamma $ is called the \textit{object change diagram} of $\cC $.

Assume that $\cC$ is connected.
Then the object change diagram of $\cC $ is either a chain or a cycle.
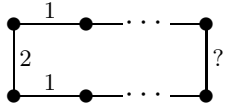
\begin{figure}
  \setlength{\unitlength}{.8mm}
  \begin{picture}(36,18)(0,3)
    \put(1,2){\circle*{2}}
    \put(2,2){\line(1,0){10}}
    \put(13,2){\circle*{2}}
    \put(14,2){\line(1,0){5}}
    \put(23,1){\makebox[0pt]{$\cdots $}}
    \put(27,2){\line(1,0){5}}
    \put(33,2){\circle*{2}}
    \put(1,3){\line(0,1){10}}
    \put(33,3){\line(0,1){10}}
    \put(1,14){\circle*{2}}
    \put(2,14){\line(1,0){10}}
    \put(13,14){\circle*{2}}
    \put(14,14){\line(1,0){5}}
    \put(23,13){\makebox[0pt]{$\cdots $}}
    \put(27,14){\line(1,0){5}}
    \put(33,14){\circle*{2}}
    \put(7,3){\makebox[0pt]{\scriptsize $1$}}
    \put(7,15){\makebox[0pt]{\scriptsize $1$}}
    \put(3,7){\makebox[0pt]{\scriptsize $2$}}
    \put(35,7){\makebox[0pt]{\scriptsize $?$}}
  \end{picture}
  \caption{Cycle diagram}
  \label{fi:cycle}
\end{figure}
In the second case the cardinality of $A$ is even, see Figure~\ref{fi:cycle}.
Let now $n\in \NN$
and assume that the object change diagram of $\cC $ is a cycle with $2n$
vertices.

Let $i,j\in I$ and $a\in A$. In analogy to \cite[1.1.6]{b-GP-00}, for all
$m\in \NN $ let
\[ \Prod(m;a,i,j)=\ldots \s_i \s_j \s^a_i \]
where the number of factors is $m$.
Let $\tau = \begin{pmatrix} 0 & 1 \\ 1 & 0 \end{pmatrix}$.

\begin{defin}
  For a $K$-Cartan scheme $\cC$ (of rank two, i.e. $|I|=2$)
  we call $\Wg(\cC)$ a {\it finite reflection groupoid} if
  the object change diagram of $\cC $ is a cycle and
\begin{eqnarray}\label{axiomF}
    \Prod(n;a,i,j) = -\tau^n
  \end{eqnarray}
  for all $a\in A$ and $i,j\in I$ with $i\not=j$, where $n=|A|/2$.
\end{defin}

Although we have no roots here, axiom \eqref{axiomF} is reasonable
since it generalizes the fact that a ``longest word'' should map
all positive roots to negative ones.

For the remaining part of this section assume that $\cC$
is a $K$-Cartan scheme such that $\Wg(\cC)$
is a finite reflection groupoid with $|A|=2n$ objects.
For simplicity assume that $I=\{1,2\}$.

By (C2), the $K$-Cartan scheme $\cC $ is locally of the form
{\tiny
\[ \xymatrix{
\cdots \ar@{-}[r]^-{1} &
{\begin{pmatrix} 2 & -c_1 \\ -c_2 & 2 \end{pmatrix}}
\ar@{-}[r]^{2} &
{\begin{pmatrix} 2 & -c_3 \\ -c_2 & 2 \end{pmatrix}}
\ar@{-}[r]^{1} &
{\begin{pmatrix} 2 & -c_3 \\ -c_4 & 2 \end{pmatrix}}
\ar@{-}[r]^-{2} & \cdots
} \]
}
for some $c_1,c_2,\ldots \in K$. Hence giving the sequence of Cartan
entries $c_1,c_2,\ldots$ and the label of the first morphism in this sequence
completely determines the $K$-Cartan scheme (and the reflection groupoid).
More precisely (compare \cite[Prop.~6.5]{p-CH09-2}):
\begin{lemma}\label{cent_sym}
  Let $a\in A$ and $i,j\in I$ with $I=\{i,j\}$.
  Let $a_1,a_2,\dots, a_{2n}\in A$ and $c_1,c_2,\dots,c_{2n}\in K$
  such that
\begin{align*}
    a_1=&a, & a_2=&\rfl _i(a_1), & a_3=&\rfl _j(a_2), &
    a_4=&\rfl _i(a_3), & a_5=&\rfl _j(a_4), & \ldots \\
    c_1=&-c^{a_1}_{ij}, & c_2=&-c^{a_2}_{ji}, &
    c_3=&-c^{a_3}_{ij}, & c_4=&-c^{a_4}_{ji}, &
    c_5=&-c^{a_5}_{ij}, & \ldots
\end{align*}
Then $c_{n+r}=c_r$ for all $r\in \{1,2,\ldots ,n\}$.
\end{lemma}
\begin{proof}
By Equation~\eqref{axiomF},
\begin{eqnarray*}
(\s \Prod(n-1;\rho_i(a),j,i)) &=&
\Prod(n;\rho_i(a),j,i) \\
&=& -\tau^n \\
&=& \Prod(n;a,i,j) \\
&=& (\Prod(n-1;\rho_i(a),j,i)\s^a_i)
\end{eqnarray*}
in $\Aut (K^I)$,
where $\s$ is the first map in $\Prod(n;\rho_i(a),j,i)$. Hence
$\s=\s^a_i$ or $\tau \s \tau=\s^a_i$
depending on whether $n$ is even or odd.
But in both cases the $c_1$ and $c_{n+1}$ defining the reflections are equal.
Starting at the other objects gives $c_{n+r}=c_r$ for all $r\in \{1,2,\ldots ,n\}$.
\end{proof}
Let
\begin{equation}
  \Phi :I\times A\to K^n, \quad (i,a) \mapsto (c_1,c_2,\ldots ,c_n),
  \label{map_Phi}
\end{equation}
where $c_1,c_2,\ldots,c_n\in K$ are depending on $i\in I$ and $a\in A$
as in Lemma~\ref{cent_sym}.

\section{Classification of finite rank two Weyl groupoids}\label{class_Weyl2}

In \cite{p-CH09-2} we gave an inductive method to construct all finite
Weyl groupoids of rank two which admit a root system.
Here, we extend these results and give a
natural bijection to triangulations of convex polygons.
We describe in Proposition~\ref{rsrank2}
the set of positive roots in terms of ordered sequences of vectors in
$\NN _0^I$. This sequences are closely related to the well-known Farey series,
see the discussion in Remark~\ref{Farey}.
As a corollary we obtain
that any non-simple positive root is a
sum of two positive roots.

\subsection{Connected simply connected Cartan schemes}

Let $\cC=\cC(I$, $A$, $(\rfl _i)_{i \in I}$, $(\Cm ^a)_{a \in A})$ be a
connected simply connected Cartan scheme (see Rem.~\ref{weyl_groupoid}) with $I=\{1,2\}$
and assume that $\rsC \re (\cC )$ is a finite root system (of type $\cC $).
By \cite[Section 6]{p-CH09-2} the object change diagram of $\cC $
is a cycle with an even number of vertices.
For all $x\in \ZZ $ let
\begin{align}
   \eta (x)=
  \begin{pmatrix}
    x & -1 \\ 1 & 0
  \end{pmatrix}.
  \label{eq:etadef}
\end{align}
Following \cite[Def.~5.1]{p-CH09-2} let
$\cA $ denote the set of finite sequences $(c_1,\ldots ,c_n)$ of
integers such that $n\ge 1$ and $\eta (c_1)\cdots \eta (c_n)=-\id $.
Let $\cAp $ be the subset of $\cA $ formed by those $(c_1,\ldots ,c_n)\in
\cA $, for which $c_i\ge 1 $ for all $i\in \{1,\ldots ,n\}$ and
the entries in the first column of $\eta (c_1)\cdots \eta (c_i)$ are
non negative for all $i<n$.

In \cite[Prop.~6.5]{p-CH09-2} it was shown that Lemma~\ref{cent_sym}
holds for $\cC $, and that $(c_1,\dots,c_n)\in \cAp $
(using the notation of the lemma).

By \cite[Prop.~5.3]{p-CH09-2} (2),(3) the dihedral group $\DG _n$ of $2n$
elements, where $n\in \NN $, acts on sequences of length $n$ in $\cAp $
by cyclic permutation of the entries and by reflections. This action gives
rise to an equivalence relation $\sim $ on $\cAp $ by taking the
orbits of the action as equivalence classes.

Let $n,m\in \NN $ with $m\ge n$, and let $c=(c_1,\ldots ,c_n)$,
$d=(d_1,\ldots ,d_m)\in \cAp $.
We write $c\approx 'd$ if and only if
\begin{itemize}
  \item $m=n$, $c\sim d$ or
  \item $m=n+1$, $d=(c_1+1,1,c_2+1,c_3,c_4,\ldots  ,c_n)$.
\end{itemize}
The following is \cite[Def.~5.4]{p-CH09-2}:
  Let $c,d\in \cAp $. Write $c\approx d$ if and only if there exists $k\in
  \NN $ and a sequence $c=e_1,e_2,\ldots ,e_k=d$ of elements of $\cAp $, such
  that $e_i\approx 'e_{i+1}$ or $e_{i+1}\approx 'e_i$ for all $i\in
  \{1,2,\ldots ,k-1\}$.

\begin{theor}\label{th:rep111} \cite[Thm.~5.5]{p-CH09-2}
  The only element of $\cAp /{\approx }$ is $(1,1,1)$.
\end{theor}

Summarizing this, the set $\cAp$ may also be defined as follows:
\begin{defin}\label{eta_seq}
Define {\it $\eta$-sequences} recursively in the following way:
\begin{enumerate}
\item $(1,1,1)$ is an $\eta$-sequence. \label{eta1}
\item If $(c_1,\ldots,c_n)$ is an $\eta$-sequence, then
$(c_2,c_3,\ldots,c_{n-1},c_n,c_1)$ and
$(c_n,c_{n-1},\ldots,c_2,c_1)$ are $\eta$-sequences. \label{eta2}
\item If $(c_1,\ldots,c_n)$ is an $\eta$-sequence, then
$(c_1+1,1,c_2+1,c_3,\ldots,c_n)$ is an $\eta$-sequence. \label{eta3}
\item Every $\eta$-sequence is obtained recursively by 
\eqref{eta1},\eqref{eta2},\eqref{eta3}.
\end{enumerate}
Then $\cAp$ is the set of $\eta$-sequences.
\end{defin}

Recall the map $\Phi$ from Equation~\eqref{map_Phi}.
\begin{theor} \cite[Thm.~6.6]{p-CH09-2} \label{th:ctors}
  Let $n\in \NN $ and $c=(c_1,c_2,\ldots ,c_n)\in \cAp $. Then there is a
  unique (up to equivalence) connected simply connected Cartan scheme of
  rank two
  such that $\rsC \re (\cC )$ is an irreducible root system and
  $c\in \mathrm{Im}\,\Phi $.
\end{theor}

Together with \cite[Prop.~6.5]{p-CH09-2} this gives that
the connected simply connected Cartan schemes of rank two, for which
$\rsC \re (\cC )$ is an irreducible root system, are those with
cycle diagram as object change diagram, and where the Cartan entries
yield an $\eta $-sequence.
Since the symmetry group of such object change diagrams is the
dihedral group, we obtain exactly one equivalence class of connected simply
connected Cartan schemes for each element of $\cAp/\sim$.

Let $\cT_n$ be the set of triangulations of a convex $n$-gon by non-inter\-secting
diagonals for a fixed labeling of the vertices by $1,\ldots,n$,
and let $\cT=\bigcup_{n=3}^\infty \cT_n$.

Let $\Psi : \cAp \rightarrow \cT$ be the map that maps an $\eta$-sequence
$(c_1,\ldots,c_n)$ to the triangulation of the $n$-gon which has $c_i$
triangles attached to the vertex $i$:
The sequence $(1,1,1)$ maps to the triangle.
Inserting a `$1$' in the $\eta $-sequence
as in Definition~\ref{eta_seq}(3) corresponds
to attaching a new triangle at the edge $(1,2)$.
Definition~\ref{eta_seq}(2)
just corresponds to the natural action of the dihedral group on the polygon.
Thus $\Psi$ is well-defined.

We can now deduce the result of this subsection:

\begin{theor}\label{wgr2_tri}
The map $\Psi$ is bijective.
\end{theor}

\begin{proof}
Since an element of $\cT$ uniquely defines an $\eta$-sequence, $\Psi$ is
injective.
On the other hand, each triangulation of the $n$-gon by non-intersecting
diagonals may be obtained by starting with a triangle and by successively
attaching triangles at the edges (this is sometimes called a flexagon or a
planar $2$-tree), so $\Psi$ is surjective.
\end{proof}

We end this subsection by determining the roots of the
Cartan schemes of rank two for which
$\rsC \re (\cC )$ is a finite irreducible root system.

\begin{defin}\label{R_seq}
Define {\it $\cEs$-sequences} as finite sequences of length $\ge 3$
with entries in $\NN _0^2$ given by the following recursion.
\begin{enumerate}
\item $((0,1),(1,1),(1,0))$ is an $\cEs$-sequence. \label{R1}
\item If $(v_1,\ldots,v_n)$ is an $\cEs$-sequence, then
\[ (v_1,\ldots,v_i,v_i+v_{i+1},v_{i+1},\ldots,v_n) \]
are $\cEs$-sequences for $i=1,\ldots,n-1$. \label{R2}
\item Every $\cEs$-sequence is obtained recursively by 
\eqref{R1} and \eqref{R2}.
\end{enumerate}
\end{defin}

\begin{remar} \label{Farey}
Notice the resemblance to Farey series (see \cite[III]{HW}):
The Farey series $\cEs_n$ of order $n$ is the ascending series of
irreducible fractions between $0$ and $1$ whose denominators do not
exceed $n$. A characteristic property is: If $\frac{h}{k}$,
$\frac{h''}{k''}$, and $\frac{h'}{k'}$ are three successive terms of $\cEs_n$, then
$\frac{h''}{k''} = \frac{h+h'}{k+k'}$.

Except $(1,0)$, we may view an entry $(a,b)$ of an $\cEs$-sequence
as the rational number $\frac{a}{b}$. This maps an $\cEs$-sequence to
an ascending sequence of rational numbers.
Axiom \eqref{R2} corresponds to the above characteristic property of a Farey series.
On positive roots (including $(1,0)$), we define
\[ (a,b) \le_\QQ (c,d) \quad \Longleftrightarrow \quad ad \le cb. \]
\end{remar}

\begin{propo} \label{rsrank2}
To each $\cEs$-sequence $f$ there is up to equivalence a unique
connected simply connected Cartan scheme $\cC$ such that $\rsCre(\cC)$
is a finite root system and the set of entries of $f$ is $R^a_+$
for some $a\in A$.

Conversely, let $\cC = \cC (I,A,(\rfl _i)_{i \in I}, (\Cm ^a)_{a \in A})$
be a Cartan scheme of rank two. Assume that $\rsCre(C)$ is a finite
irreducible root system. Let $a\in A$. Then sorting the elements
of $R^a\cap \NN _0^2$
with respect to $\le_\QQ$ in ascending order yields an
$\cEs$-sequence.
\end{propo}

\begin{proof}
For $c\in\cAp$ of length $n$, consider
the Cartan scheme $\cC $ with objects
$a_1,\ldots,a_{2n}$ such that $\Phi(2,a_1)=c$ and $\rho_2(a_1)=a_2$.
The numbers $c_1,\ldots,c_n$ are the Cartan entries used for certain
reflections $s_i\in \Hom (a_{i+1},a_i)$, $1\le i \le 2n$, where $a_{2n+1}=a_1$.
By the proof of \cite[Thm.~6.6]{p-CH09-2}, the real positive roots at $a_1$ are
\[ \{ \beta_\nu:= s_1 s_2 \cdots s_\nu (\alpha_{2-(\nu \Mod 2)})
\:\:|\:\: \nu=0,\ldots,n-1 \}, \]
or equivalently
\[ \{ \tilde\beta_\nu:=s_{2n}^{-1}s_{2n-1}^{-1}\cdots s_{2n+1-\nu }^{-1}
(\alpha_{(\nu \Mod 2)+1}) \:\:|\:\: \nu=0,\ldots,n-1 \}. \]
For $0\le\mu<n$, the equation
$s_1\ldots s_{\mu+1}(-\tau^n)=s_{2n}^{-1}\ldots s_{n+\mu+2}^{-1}$ implies
that $\beta_\mu = \tilde\beta_{n-1-\mu}$, thus the set of real positive roots
is also
\[ \{\beta_0,\ldots,\beta_\nu,\tilde\beta_{n-2-\nu},\ldots,\tilde\beta_0 \} \]
for any fixed $0\le\nu<n$.
If
using Definition~\ref{eta_seq}\eqref{eta2},\eqref{eta3}
we include a `$1$'
in the sequence
at position $\nu$ (and of course also at position $\nu+n$), i.e.
\[ (\ldots,c_{\nu},c_{\nu+1},\ldots) \mapsto
(\ldots,c_{\nu}+1,1,c_{\nu+1}+1,\ldots), \]
this will not affect the roots
$\tilde\beta_{n-2-\nu},\ldots,\tilde\beta_0,\beta_0,\ldots,\beta_{\nu-1}$
coming from $a_{n+3+\nu},\ldots,a_{2n},a_1,a_1,\ldots,a_\nu$.
Let $i,j\in I$ be such that $\rho_i(a_\nu)=a_{\nu+1}$ and $i\ne j$.
If $\tilde s_\nu$, $t$, $\tilde s_{\nu+1}$ are the new morphisms
corresponding to $c_{\nu}+1$, $1$, $c_{\nu+1}+1$ respectively, then
\[ s_1\cdots s_{\nu-1}\tilde s_\nu(\alpha_j)
= s_1\cdots s_\nu(\alpha_j)+s_1\cdots s_{\nu-1}(\alpha_i)
= \beta_\nu+\beta_{\nu-1}, \]
\[ s_1\cdots s_{\nu-1} \tilde s_\nu t(\alpha_i)=\beta_\nu. \]
These will be the positive roots coming from the two objects replacing $a_{\nu+1}$.
Hence $\beta_\nu+\beta_{\nu-1}$ is the new root after the transformation.
Since the positive roots for the $\eta$-sequence $(1,1,1)$ are
$((0,1),(1,1),(1,0))$, this proves that $\cEs$-sequences are sets of
positive roots.

Conversely, a set of positive roots uniquely determines an $\eta$-sequence
by \cite[Prop.~6.5]{p-CH09-2}, and this $\eta $-sequence
corresponds to an $\cEs$-sequence by the above construction.
\end{proof}

\begin{corol} \label{sumof2}
Let $\cC = \cC (I,A,(\rfl _i)_{i \in I}, (\Cm ^a)_{a \in A})$
be a Cartan scheme of rank two.
If $\rsCre(\cC) = (\cC,(R^a)_{a\in A})$ is a finite root system of type $\cC $
then, for all $a\in A$ and $\alpha\in R^a_+$, either $\alpha$ is simple
or it is the sum of two positive roots in $R^a_+$.
\end{corol}

\begin{examp}
    We determine the sets $R^a_+$ for all
    Cartan schemes $\cC $ where $\rsC \re (\cC )$
    is a finite irreducible root system with at most 5 positive roots.
\begin{enumerate}
\item With three positive roots:
    The $\eta$-sequence is $(1,1,1)$, it corresponds to the universal
    covering of the Weyl group of type $A_2$ (see Definition~\ref{de:cover}).
    The set of positive roots is $\{(0,1)$, $(1,1)$, $(1,0)\}$.
\item With four positive roots: There are two positions for including
a `$1$', one obtains $(2,1,2,1)$ or $(1,2,1,2)$ which are equivalent with
respect to $\sim$. The sets of positive roots are (type $B_2$ and $C_2$):
$\{(0,1),(1,2),(1,1),(1,0)\}$, $\{(0,1),(1,1),(2,1),(1,0)\}$.
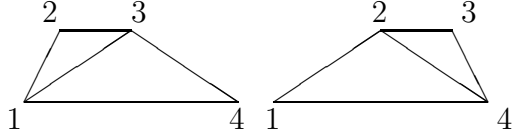
\begin{center}
\setlength{\unitlength}{2947sp}%
\begingroup\makeatletter\ifx\SetFigFont\undefined%
\gdef\SetFigFont#1#2#3#4#5{%
  \reset@font\fontsize{#1}{#2pt}%
  \fontfamily{#3}\fontseries{#4}\fontshape{#5}%
  \selectfont}%
\fi\endgroup%
\begin{picture}(4217,1035)(451,-1786)
\thinlines
{\put(601,-1561){\line( 3, 2){900}}
\put(1501,-961){\line( 3,-2){900}}
\put(2401,-1561){\line(-1, 0){1800}}
}%
{\put(601,-1561){\line( 1, 2){300}}
\put(901,-961){\line( 1, 0){600}}
}%
{\put(2701,-1561){\line( 3, 2){900}}
\put(3601,-961){\line( 3,-2){900}}
\put(4501,-1561){\line(-1, 0){1800}}
}%
{\put(3601,-961){\line( 1, 0){600}}
\put(4201,-961){\line( 1,-2){300}}
}%
\put(751,-886){\makebox(0,0)[lb]{\smash{{\SetFigFont{12}{14.4}{\rmdefault}{\mddefault}{\updefault}{2}%
}}}}
\put(1501,-886){\makebox(0,0)[lb]{\smash{{\SetFigFont{12}{14.4}{\rmdefault}{\mddefault}{\updefault}{3}%
}}}}
\put(451,-1786){\makebox(0,0)[lb]{\smash{{\SetFigFont{12}{14.4}{\rmdefault}{\mddefault}{\updefault}{1}%
}}}}
\put(2326,-1786){\makebox(0,0)[lb]{\smash{{\SetFigFont{12}{14.4}{\rmdefault}{\mddefault}{\updefault}{4}%
}}}}
\put(2626,-1786){\makebox(0,0)[lb]{\smash{{\SetFigFont{12}{14.4}{\rmdefault}{\mddefault}{\updefault}{1}%
}}}}
\put(4576,-1786){\makebox(0,0)[lb]{\smash{{\SetFigFont{12}{14.4}{\rmdefault}{\mddefault}{\updefault}{4}%
}}}}
\put(4276,-886){\makebox(0,0)[lb]{\smash{{\SetFigFont{12}{14.4}{\rmdefault}{\mddefault}{\updefault}{3}%
}}}}
\put(3526,-886){\makebox(0,0)[lb]{\smash{{\SetFigFont{12}{14.4}{\rmdefault}{\mddefault}{\updefault}{2}%
}}}}
\end{picture}%
\end{center}
\item With five positive roots: There is only one triangulation of a convex
    pentagon up to operation of the dihedral group.
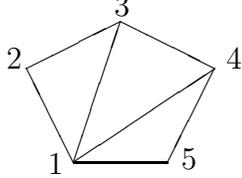
\begin{center}
\setlength{\unitlength}{1947sp}%
\begingroup\makeatletter\ifx\SetFigFont\undefined%
\gdef\SetFigFont#1#2#3#4#5{%
  \reset@font\fontsize{#1}{#2pt}%
  \fontfamily{#3}\fontseries{#4}\fontshape{#5}%
  \selectfont}%
\fi\endgroup%
\begin{picture}(2792,2160)(451,-2311)
\thinlines
{\put(1201,-2161){\line(-1, 2){600}}
\put(601,-961){\line( 2, 1){1200}}
\put(1801,-361){\line( 2,-1){1200}}
\put(3001,-961){\line(-1,-2){600}}
\put(2401,-2161){\line(-1, 0){1200}}
}%
{\put(1201,-2161){\line( 3, 2){1800}}
}%
{\put(1201,-2161){\line( 1, 3){600}}
}%
\put(876,-2311){\makebox(0,0)[lb]{\smash{{\SetFigFont{12}{14.4}{\rmdefault}{\mddefault}{\updefault}{1}%
}}}}
\put(351,-961){\makebox(0,0)[lb]{\smash{{\SetFigFont{12}{14.4}{\rmdefault}{\mddefault}{\updefault}{2}%
}}}}
\put(3151,-961){\makebox(0,0)[lb]{\smash{{\SetFigFont{12}{14.4}{\rmdefault}{\mddefault}{\updefault}{4}%
}}}}
\put(2576,-2236){\makebox(0,0)[lb]{\smash{{\SetFigFont{12}{14.4}{\rmdefault}{\mddefault}{\updefault}{5}%
}}}}
\put(1726,-286){\makebox(0,0)[lb]{\smash{{\SetFigFont{12}{14.4}{\rmdefault}{\mddefault}{\updefault}{3}%
}}}}
\end{picture}%
\end{center}
The corresponding sets of positive roots are
\begin{align*}
    &\{(0,1),(1,3),(1,2),(1,1),(1,0)\}
    \text{ (assoc. to the given triangulation)},\\
    &\{(0,1),(1,1),(2,1),(3,1),(1,0)\},
    \{(0,1),(1,2),(2,3),(1,1),(1,0)\},\\
    &\{(0,1),(1,2),(1,1),(2,1),(1,0)\},
    \{(0,1),(1,1),(3,2),(2,1),(1,0)\}.
\end{align*}
\end{enumerate}
\end{examp}

\subsection{Finite rank two Weyl groupoids}

After considering simply connected Cartan schemes, we may
now investigate the quotients of these universal coverings.
Coverings of Cartan schemes have been treated in \cite[Sect.~3]{p-CH09-2}.
We recall the concept from \cite[Def.~3.1]{p-CH09-2}:

\begin{defin} \label{de:cover}
  Let $\cC =\cC (I,A,(\rfl _i)_{i\in I},(C^a)_{a\in A})$ and
  $\cC' =\cC' (I,A'$, $(\rfl '_i)_{i\in I},(C'^a)_{a\in A'})$ be
  connected Cartan schemes.
  Let $\pi :A'\to A$ be a map such that $C^{\pi (a)}=C'^a$
  for all $a\in A'$ and the diagrams
  \begin{align}
    \begin{CD}
      A' & @>\rfl '_i>> & A' \\
      @V\pi VV & & @VV\pi V\\
      A & @>>\rfl _i > & A
    \end{CD}
    \label{eq:pirfl}
  \end{align}
  commute for all $i\in I$. We say that $\pi :\cC '\to \cC $
  is a \textit{covering}, and that $\cC '$ is a \textit{covering of} $\cC $.
\end{defin}

By \cite[Prop.~3.4]{p-CH09-2} and \cite[Cor.~3.6]{p-CH09-2},
coverings behave similarly as for example in topology:
If $\cC$ is a Cartan scheme and $a$ an object, then for each subgroup
$U\le \End(a)$ there exists up to equivalence a unique Cartan scheme $\cC'$,
a covering $\pi : \cC'\rightarrow \cC$ and an object $b$ of $\cC'$ such
that $\End(b)$ and $U$ are isomorphic via $\pi$ (see the definition
of the covariant functor $F_\pi$ in \cite[Sect.~3]{p-CH09-2}).
In particular if $U$ is the trivial subgroup, then we obtain the
universal covering.

Let $n\in \NN $ and let
$\cC$ be a connected simply connected Cartan scheme of rank two with
finite root system and $2n$ objects.
By Definition~\ref{de:cover}, to obtain all coverings $\cC\rightarrow \cC'$,
we have to determine the pairs of objects $a,b$ of $\cC$ for which
$\Phi(i,a)=\Phi(i,b)$ for some (equivalently all) $i\in I$.

Consider the action of the dihedral group $\DG_{2n}$ on the
object change diagram of $\cC $ and write
$\DG_{2n} = Z_{2n}\rtimes Z_2$, where $Z_k$ is the cyclic group
with $k$ elements; let $\tau$ be the generator of the $Z_2$ part
on the right and $\mu$ a generator of $Z_{2n}$ rotating
the object change diagram by one generating morphism.
Let $H$ be the following subgroup of $\DG_{2n}$:
\begin{align}
  H = \{ \mu^{2k}, \tau\mu^{2k+1} \in \DG_{2n} \mid k\in\ZZ\}.
  \label{eq:H}
\end{align}

The group $\End(a')$ for an object $a'$ of $\cC'$ will at most be a subgroup
of $H$.
Indeed, let $i\in I$ and $a,b\in A$ with $\Phi (i,a)=\Phi (i,b)$. Then
either the distance between $a$ and $b$ in the object change diagram is even,
then we just rotate the object change diagram; or the distance is odd,
then we rotate and change direction.

We determine first the periodicity of an $\eta$-sequence.
If $c$ is a sequence, let $c^r$ denote the $r$ times repeated
sequence $c$.
\begin{propo}\label{eta_pow}
Let $c$ be an $\eta$-sequence. Write $c=(a_1,\ldots,a_n)^r$ for
suitable $r,n,a_1,\ldots,a_n\in\NN$. Then $r\in \{1,2,3\}$.
\end{propo}
\begin{proof}
Observe first that by Definition~\ref{eta_seq}
an $\eta$-sequence always contains a $1$.
If $n=1$, then $c=(1,1,1)$ again by Definition~\ref{eta_seq},
so $r=3$.

Assume that $n=2$.
Then $c=(a,1)^r$ for some $a\in\NN$ and $a\ne 1$ since
two consecutive $1$'s may never appear by construction unless
the sequence is $(1,1,1)$.
If $r=2k$ then
\[ c=(a,1,a,1)^k\approx (a-1,a-1,1)^k \sim (a-1,1,a-1)^k. \]
Thus if $a-1=1$ then $r=2$ and $c=(2,1,2,1)$. Otherwise $a>2$ and
$(a-1,1,a-1)^k \approx (a-2,a-2)^k$
which is a contradiction.
If $r=2k+1$ for some $k\ge 1$,
then $c=(a,1,a,1)^k(a,1)\approx (a-1,a-1,1)^k(a,1)$
and therefore $a>2$. The sequence $(a-1,a-1,1)^k(a,1)$ is equivalent to
\[ (a-1,1,a-1)^{k-1}(a-1,1,a,1,a-1) \approx (a-2)^{2k+1} \]
and hence $a=3$, $k=1$, i.e.~$c=(3,1,3,1,3,1)$ and $r=3$.

Assume now that $n>2$. Then without loss of generality $a_2=1$ and
$c \approx (a_1-1,a_3-1,a_4,\ldots,a_n)^r$ has length $r(n-1)$.
Hence we conclude by induction on $n$ that $r\in \{1,2,3\}$.
\end{proof}

Call an $\eta$-sequence $c=(a_1,\ldots,a_n)$ of ``{\it type 1}'' if there exists
$k$ such that $(a_k,\ldots,a_n,a_1,\ldots,a_{k-1})=
(a_k,a_{k-1},\ldots,a_1,a_n,\ldots,a_{k+1})$.
Otherwise, we say that $c$ is of ``{\it type 2}''. We obtain:

\begin{propo}\label{end_aa}
Let $\cC$ be a connected simply connected Cartan scheme of
rank two. Assume that $\rsC \re (\cC )$ is a finite irreducible root system.
Let $a\in A$ and $c=\Phi(1,a)$.
Let $\pi : \cC\rightarrow \cC'$ be the covering for which
$\End(\pi(a))$ is maximal.
Write $c=d^r$ for a sequence $d$ and $r\in\NN$ such that $r$ is maximal.
Let $m=|A|/2r$ be the length of $d$.
Then $\End(\pi(a))$ is given by the following subgroup
of $H$ (see Equation~\eqref{eq:H}):
\begin{center}
\begin{tabular}{ l | c c c }
$m$ even:\\
$c$ & $r=1$ & $r=2$ & $r=3$ \\
\hline
type 1 & $Z_2\times Z_2$ & $\DG_4$ & $\DG_6$ \\
type 2 & $Z_2=\langle\mu^n\rangle$ & $Z_4=\langle\mu^\frac{n}{2}\rangle$ &
$Z_6=\langle\mu^\frac{n}{3}\rangle$ \vspace{5pt}\\
$m$ odd:\\
$c$ & $r=1$ & $r=2$ & $r=3$ \\
\hline
type 1 & $Z_2$ & $Z_2\times Z_2$ & $Z_3 \rtimes Z_2$ \\
type 2 & $1$ & $Z_2=\langle\mu^n\rangle$ & $Z_3=\langle\mu^\frac{2n}{3}\rangle$
\end{tabular}
\end{center}
\end{propo}

\begin{proof}
Remember that if the $\eta$-sequence $c$ is of length $n$ then
we have $2n$ objects and the Cartan entries used for the $2n$
generating morphisms are the entries of $c^2$.
So by Proposition~\ref{eta_seq}, what we are looking for are the elements
$\sigma$ of $H$ fixing the sequence $c^2=d^{2r}$ where
$r\in\{1,2,3\}$.
\end{proof}

\begin{remar}
To construct all Cartan schemes for which $\cC$ is a covering,
proceed in the following way: Choose a subgroup $U$ of the group
given in Proposition~\ref{end_aa}, and view it as a subgroup of $\DG_{2n}$
as explained above Proposition~\ref{eta_pow},
so $U$ acts on the set $A$ of objects of $\cC$.
Define $A'$ to be the set of orbits of $U$ on $A$, and
$\pi : A \rightarrow A'$ maps an object to its orbit.
Then set $C'^{\pi(a)}=C^a$ and $\rho'_i$ are given by the
commutative diagram in Definition~\ref{de:cover}.
The quadruple $\cC' =\cC' (I,A'$, $(\rfl '_i)_{i\in I},(C'^a)_{a\in A'})$
is then a Cartan scheme with finite root system and
$\cC \rightarrow \cC'$ is a covering as defined in Definition~\ref{de:cover}.
\end{remar}

\section{Finite reflection groupoids and the Grassmannians}\label{rank_two}

Let $n\in \NN _{\ge 3}$.
Recall from \cite{FR07} that the cluster algebra of type $A_n$ is
constructed in the following way: Start with a convex $n$-gon and
choose a triangulation by non-intersecting diagonals.
Label these diagonals by $x_1,\ldots,x_{n-3}$,
the edges by $q_1,\ldots,q_n$ and the vertices of the $n$-gon
by $1,\ldots,n$.
The labels for the remaining diagonals are obtained by flipping diagonals;
two adjacent triangles $(x,a,b)$ and $(x,c,d)$ yield a
relation for the other diagonal $y$ (see Figure \ref{fi:flip}):
\[ xy = ac+bd. \]
\begin{figure}
\setlength{\unitlength}{3047sp}%
\begingroup\makeatletter\ifx\SetFigFont\undefined%
\gdef\SetFigFont#1#2#3#4#5{%
  \reset@font\fontsize{#1}{#2pt}%
  \fontfamily{#3}\fontseries{#4}\fontshape{#5}%
  \selectfont}%
\fi\endgroup%
\begin{picture}(3392,2485)(1576,-3061)
\thinlines
{\put(1801,-961){\line( 1, 0){3000}}
\put(4801,-961){\line( 0,-1){1800}}
\put(4801,-2761){\line(-1, 0){3000}}
\put(1801,-2761){\line( 0, 1){1800}}
\put(1801,-961){\line( 5,-3){3000}}
}%
{\put(4801,-961){\line(-5,-3){3000}}
}%
\put(1576,-1786){\makebox(0,0)[lb]{\smash{{\SetFigFont{12}{14.4}{\rmdefault}{\mddefault}{\updefault}{$a$}%
}}}}
\put(3226,-811){\makebox(0,0)[lb]{\smash{{\SetFigFont{12}{14.4}{\rmdefault}{\mddefault}{\updefault}{$b$}%
}}}}
\put(4876,-1936){\makebox(0,0)[lb]{\smash{{\SetFigFont{12}{14.4}{\rmdefault}{\mddefault}{\updefault}{$c$}%
}}}}
\put(3226,-3061){\makebox(0,0)[lb]{\smash{{\SetFigFont{12}{14.4}{\rmdefault}{\mddefault}{\updefault}{$d$}%
}}}}
\put(2476,-2161){\makebox(0,0)[lb]{\smash{{\SetFigFont{12}{14.4}{\rmdefault}{\mddefault}{\updefault}{$x$}%
}}}}
\put(4051,-2161){\makebox(0,0)[lb]{\smash{{\SetFigFont{12}{14.4}{\rmdefault}{\mddefault}{\updefault}{$y$}%
}}}}
\end{picture}%
  \caption{Diagonal flip}
  \label{fi:flip}
\end{figure}
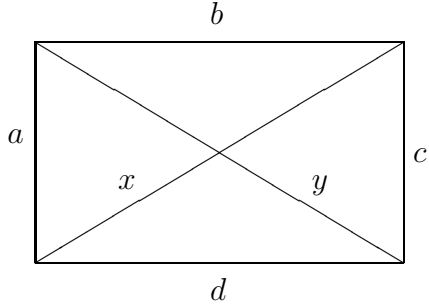
The cluster algebra of type $A_{n-3}$ is the subalgebra of the
rational function field in $x_1,\ldots,x_{n-3},q_1,\ldots,q_n$ generated by
all the labels in the $n$-gon. In a simplified version, which will be our
point of view in this text, the variables $q_1,\dots,q_n$ are specialized to
$1$.

Another description is by the Grassmannian $\Gr(2,n)$
(see \cite[Lecture 3]{FR07}):
Take a matrix $z=(z_{ij})\in K^{2\times n} $. For $1\le k<l\le n$, let
\[ P_{k,l} =
\det \begin{pmatrix} z_{1k} & z_{1l} \\ z_{2k} & z_{2l} \end{pmatrix}. \]
These are the Pl\"ucker coordinates of the row span of $z$ as an element
of the Grassmannian $\Gr(2,n)$. They satisfy the relations
\begin{equation} \label{eq:Plu}
  P_{i,k}P_{j,l} = P_{i,j}P_{k,l}+P_{i,l}P_{j,k}
\end{equation}
for all $i,j,k,l\in \{1,\dots,n\}$ with $i<j<k<l$.
By assigning the value of $P_{i,j}$ to the variable $x_k$ associated
to the segment $(i,j)$ of the $n$-gon,
it follows by induction that the rational function
associated to every diagonal is equal to the corresponding $P_{k,l}$.

Let ${\mathcal O}(\Gr(2,n))$ be the homogeneous coordinate ring
of $\Gr(2,n)$. Equivalently, it is the
commutative $K$-algebra generated by elements
$P_{i,j}$ with $1\le i<j\le n$ and the Pl\"ucker relations above. Define
$P_{i,j}=P_{j,i}$ whenever $1\le j<i\le n$, and consider the indices of the
Pl\"ucker coordinates as elements of the cyclic group $Z_n$. Let
\[ \cG(2,n)={\mathcal O}(\Gr(2,n))/I,\quad
I=(P_{i,i+1}-1\mid i=1,\ldots,n).\]

\begin{lemma}\label{grass_gen}
The algebra $\cG(2,n)$ is generated by the elements
\[ P_{1,3},P_{2,4},\ldots,P_{n-2,n}.\]
\end{lemma}
\begin{proof}
    Let $\mathcal{G}$ be the subalgebra of $\mathcal{G}(2,n)$ generated by
    the elements $P_{i,i+2}$ with $1\le i\le n-2$.
    Let $i,j\in \{1,\dots,n\}$ with $i+2<j$. Then
$P_{i,i+1} P_{i+2,j}+P_{i,j}P_{i+1,i+2} = P_{i,i+2}P_{i+1,j}$,
that is
\[ P_{i,j} = P_{i,i+2} P_{i+1,j} - P_{i+2,j}. \]
Hence $P_{i,j}\in \mathcal{G}$ for all $i,j\in \{1,\dots,n\}$ with $i+1<j$
by induction on $j-i$. This proves the claim.
\end{proof}

Let $\cC $ be a connected simply connected Cartan scheme of rank two
such that $\rsC \re (\cC )$ is a finite irreducible root system.
Assume that $n=|A|/2$ and that $I=\{1,2\}$.
Recall that
\[
  \tau =
  \begin{pmatrix}
    0 & 1 \\ 1 & 0
  \end{pmatrix}, \quad
  \eta (x)=
  \begin{pmatrix}
    x & -1 \\ 1 & 0
  \end{pmatrix}
\quad
\]
for all $x\in \ZZ $.
Then
\[ \sigma_1^a = \begin{pmatrix} -1 & -c_{12}^a \\ 0 & 1 \end{pmatrix}
    = \eta(-c_{12}^a)\tau, \]
    \[ \sigma_2^a = \begin{pmatrix} 1 & 0 \\ -c^a_{21} & -1 \end{pmatrix}
	= \tau \eta(-c^a_{21}) \]
for all $a\in A$, where the matrices are given with respect
to the standard basis of $\ZZ ^2$.
Since $\tau^2=\id$,
\begin{align*}
\Prod(n; a,1,2)\tau =& \ldots \sigma _1\sigma_2 \sigma^a_1\tau \\
=& \ldots \sigma _1\tau \tau \sigma_2 \sigma^a_1\tau \\
=& \tau^{n-1} \eta(c_n) \ldots \eta(c_2) \eta(c_1)
= -\tau^{n-1}
\end{align*}
for all $a\in A$,
where the last equation follows from \cite[Prop.\,6.5]{p-CH09-2}.
This argument shows that Axiom \eqref{axiomF} is valid
for $\cC $:

\begin{corol}
Let $\cC $ be a connected simply connected Cartan scheme of rank two.
Assume that $\rsC \re (\cC )$ is a finite irreducible root system
of type $\cC $.
Then $\Wg (\cC )$ is a finite reflection groupoid.
\end{corol}

We now formulate our main result on reflection groupoids
in purely algebraic terms. Recall that $n\in \NN _{\ge 3}$.

\begin{theor}\label{thm:main}
Let $J$ be the ideal of $R=K[c_1,\ldots,c_n]$
generated by the coefficients of the $2\times 2$-matrix
\begin{equation}\label{eta_id}
\eta(c_n) \ldots \eta(c_2) \eta(c_1) + \id.
\end{equation}
There is a unique algebra isomorphism
\[ \bar{\varphi } :\, R/J \rightarrow \cG(2,n), \quad
c_i \mapsto P_{i,i+2}. \]
The inverse of $\bar{\varphi }$ can be defined by
\[ \bar{\varphi }^{-1}(P_{i,j})=(\eta (c_{j-2})\cdots \eta (c_{i+1})
\eta (c_i))_{11} \quad \text{for all $1\le i<j\le n$,}
\]
where $(\cdot )_{11}$ denotes the entry in the first row and first column.
\end{theor}

The proof of Theorem~\ref{thm:main} uses several observations
and lemmata which are collected here.

\begin{lemma}
  \label{le:RJgen}
  Let $J$ be the ideal of $R=K[c_1,\ldots,c_n]$
  generated by the coefficients of the $2\times 2$-matrix in
  Equation~\eqref{eta_id}.
  The algebra $R/J$ is generated by the elements $c_1,\dots,c_{n-2}$.
\end{lemma}

\begin{proof}
    Since $\eta (c)$ is invertible for all $c\in R$, Equation~\eqref{eta_id}
    implies that
    \begin{align*}
      \eta (c_{n-2})\cdots \eta (c_2)\eta (c_1)=
      -(\eta (c_n)\eta (c_{n-1}))^{-1}.
    \end{align*}
    Looking at the non-diagonal entries yields that $c_{n-1}$ and $c_n$
    are polynomials in $c_1,\dots,c_{n-2}$.
\end{proof}

For all $a,b,c\in K$ let
\[ \mu(a,b,c) = \begin{pmatrix} a & -b \\ c & 0 \end{pmatrix} \in K^{2\times 2}. \]
Explicit calculation shows that (see Figure \ref{fi:equ_star} for an illustration)
\begin{equation}
  \label{equ1}
  \begin{aligned}
       & \mu(f_1,a,b)\mu(f_2,b,c)\mu(f_3,c,d) f_2 = \\
       & \qquad \mu (f_1f_2-ac,ab,bf_2)\mu (f_2f_3-bd,cf_2,cd)
  \end{aligned}
\end{equation}
for all $f_1,f_2,f_3,a,b,c,d\in K$.
Equation~\eqref{equ1} is a generalization of \cite[5.2]{p-CH09-2}.
\begin{figure}
\setlength{\unitlength}{3347sp}%
\begingroup\makeatletter\ifx\SetFigFont\undefined%
\gdef\SetFigFont#1#2#3#4#5{%
  \reset@font\fontsize{#1}{#2pt}%
  \fontfamily{#3}\fontseries{#4}\fontshape{#5}%
  \selectfont}%
\fi\endgroup%
\begin{picture}(6024,2624)(1789,-3373)
\thinlines
{\put(1801,-3361){\line( 1, 3){600}}
\put(2401,-1561){\line( 4, 1){2400}}
\put(4801,-961){\line( 4,-1){2400}}
\put(7201,-1561){\line( 1,-3){600}}
\put(7801,-3361){\line(-3, 1){5400}}
\put(2401,-1561){\line( 1, 0){4800}}
\put(7201,-1561){\line(-3,-1){5400}}
\put(1801,-3361){\line( 5, 4){3000}}
\put(4801,-961){\line( 5,-4){3000}}
\put(7801,-3361){\line( 0, 1){  0}}
}%
\put(3376,-1111){\makebox(0,0)[lb]{\smash{{\SetFigFont{12}{14.4}{\rmdefault}{\mddefault}{\updefault}{$b$}%
}}}}
\put(6076,-1111){\makebox(0,0)[lb]{\smash{{\SetFigFont{12}{14.4}{\rmdefault}{\mddefault}{\updefault}{$c$}%
}}}}
\put(1951,-2236){\makebox(0,0)[lb]{\smash{{\SetFigFont{12}{14.4}{\rmdefault}{\mddefault}{\updefault}{$a$}%
}}}}
\put(7576,-2236){\makebox(0,0)[lb]{\smash{{\SetFigFont{12}{14.4}{\rmdefault}{\mddefault}{\updefault}{$d$}%
}}}}
\put(6751,-2311){\makebox(0,0)[lb]{\smash{{\SetFigFont{12}{14.4}{\rmdefault}{\mddefault}{\updefault}{$f_3$}%
}}}}
\put(2776,-2311){\makebox(0,0)[lb]{\smash{{\SetFigFont{12}{14.4}{\rmdefault}{\mddefault}{\updefault}{$f_1$}%
}}}}
\put(5226,-2986){\makebox(0,0)[lb]{\smash{{\SetFigFont{12}{14.4}{\rmdefault}{\mddefault}{\updefault}{$\frac{f_2f_3-bd}{c}$}%
}}}}
\put(3751,-2986){\makebox(0,0)[lb]{\smash{{\SetFigFont{12}{14.4}{\rmdefault}{\mddefault}{\updefault}{$\frac{f_1f_2-ac}{b}$}%
}}}}
\put(4726,-1861){\makebox(0,0)[lb]{\smash{{\SetFigFont{12}{14.4}{\rmdefault}{\mddefault}{\updefault}{$f_2$}%
}}}}
\end{picture}%
\caption{Equation \eqref{equ1}}
\label{fi:equ_star}
\end{figure}
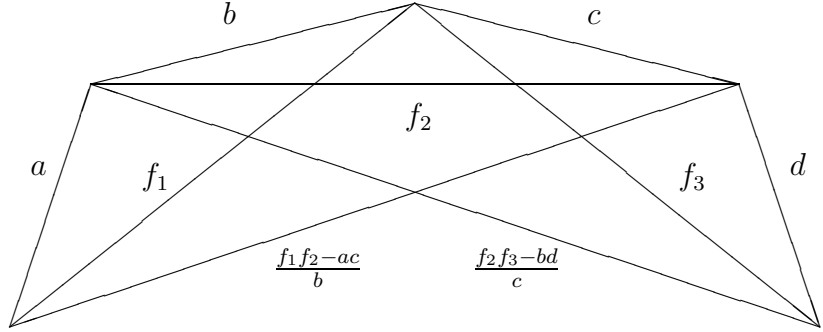

\begin{lemma}
    \label{le:murel}
Equation
\begin{equation}\label{equ2}
\quad \prod_{i=1}^n \mu(P_{i,i+2},P_{i,i+1},P_{i+1,i+2}) =
-\prod _{i=1}^n P_{i,i+1}\,\id
\end{equation}
holds in ${\mathcal O}(\Gr(2,n))^{2\times 2}$.
\end{lemma}

\begin{proof}
    We proceed by induction on $n$.
If $n=3$, then Equation~\eqref{equ2} becomes
\[ \mu(P_{1,3},P_{1,2},P_{2,3})\mu(P_{1,2},P_{2,3},P_{1,3})
\mu(P_{2,3},P_{1,3},P_{1,2})=-P_{1,2}P_{2,3}P_{1,3} \id, \]
which can be checked by explicit calculation. Similarly,
Equation~\eqref{equ2} holds for $n=4$.

Assume that $n>4$.
The subalgebra of ${\mathcal O}(\Gr(2,n))$
generated by all $P_{i,j}$ with $i,j\ne 3$ is isomorphic to
${\mathcal O}(\Gr(2,n-1))$, hence by induction hypothesis we have
\begin{equation}\label{eqn_nu}
  \begin{aligned}
  \mu(P_{1,4},P_{1,2},P_{2,4})\mu(P_{2,5},P_{2,4},P_{4,5})
  \prod_{i=4}^n \mu(P_{i,i+2},P_{i,i+1},P_{i+1,i+2})\quad &\\
  = -P_{1,2}P_{2,4}P_{4,5}\cdots P_{n,n+1} \id.&
  \end{aligned}
\end{equation}
By Equation~\eqref{equ1} and the Grassmann-Pl\"ucker relations \eqref{eq:Plu}
we have
\begin{equation} \label{eq:mm=mmm}
    \begin{aligned}
\mu(P_{1,3},P_{1,2},P_{2,3})\mu(P_{2,4},P_{2,3},P_{3,4})
\mu(P_{3,5},P_{3,4},P_{4,5}) P_{2,4} \quad &\\
= \mu(P_{1,4},P_{1,2},P_{2,4})
\mu(P_{2,5},P_{2,4},P_{4,5}) P_{2,3} P_{3,4}.&
\end{aligned}
\end{equation}
By multiplying Equation~\eqref{eqn_nu} with $P_{2,3}P_{3,4}$, inserting
Equation~\eqref{eq:mm=mmm}, and dividing by $P_{2,4}$ (${\mathcal O}(\Gr (2,n))$
is an integral domain) we obtain Equation~\eqref{equ2} which completes
the induction.
\end{proof}

\begin{lemma}
    \label{le:eta}
    Let $S$ be a ring with $1$. For all $x\in S$ define $\eta (x)\in S^{2\times 2}$
    similarly to Equation~\eqref{eq:etadef}.
    Let $x,y\in S$, $A\in S^{2\times 2}$, and $a=A_{11}$. Then
    $(\eta (x)A)_{21}=a$ and $(A\eta (y))_{12}=(\eta (x)A\eta (y))_{22}
    =-a$.
\end{lemma}

\begin{proof}[Proof of Theorem~\ref{thm:main}]
There exists a unique algebra map $\varphi :R\to \cG (2,n)$ such that
$\varphi (c_i)=P_{i,i+2}$ for all $i$.
By Lemma~\ref{grass_gen} this map is surjective.
It remains to show that $\ker \varphi =J$.
First we conclude that $\bar{\varphi }$ is well-defined,
that is, $J\subseteq \ker \varphi $.
Indeed, specializing $P_{i,i+1}$ to $1$ for all $i=1,\ldots,n$
allows us to replace the factors $\mu (P_{i,i+2},P_{i,i+1},P_{i+1,i+2})$
in Equation~\eqref{equ2} by $\eta (P_{i,i+2})$ for all $i\in \{1,2,\dots,n\}$.
Thus $\varphi (J)=0$, and hence $\bar{\varphi }$ is well-defined.

We prove the injectivity of $\bar{\varphi }$ by determining the inverse map.
First let
\begin{equation}
  \psi (P_{i,j})=
  (\eta (c_{j-2})\cdots \eta (c_{i+1})
  \eta (c_i))_{11}\in R/J
  \label{eq:psi}
\end{equation}
for all $1\le i<j\le n$.
Then $\psi (P_{i,i+1})=1$ for all $i\in \{1,2,\dots,n-1\}$
and $\psi (P_{i,i+2})=c_i$ for all $i\in \{1,2,\dots,n-2\}$.
Lemma~\ref{le:eta} implies that
\begin{equation}
    \label{eq:psiprop1}
    \begin{aligned}
  \psi (P_{j,k+1})=&\:\sum _{r=1}^2\eta (c_{k-1})_{1r}
        (\eta (c_{k-2})\cdots \eta (c_{j+1})\eta (c_j))_{r1}
    \\
    =&\:\: c_{k-1}\psi (P_{j,k})-\psi (P_{j,k-1})
\end{aligned}
\end{equation}
for all $1\le j<k<n$,
where the potential summand $\psi (P_{j,j})$ is considered as $0$.
Using Equation~\eqref{eq:psiprop1}, one shows by induction on $k$ that
\begin{equation}
    \psi (P_{i,j}) = \psi (P_{i,k})\psi (P_{j,k+1})
    -\psi (P_{j,k})\psi (P_{i,k+1})
    \label{eq:psiprop2}
\end{equation}
for all $1\le i<j\le k<n$.

Let $\psi :{\mathcal O}(\Gr(2,n))\to R/J$ be the unique algebra map where
$\psi (P_{i,j})$ is as in Equation~\eqref{eq:psi}.
We prove that $\psi $ is well-defined.
Let $i,j,k,l\in \{1,2,\dots,n\}$ with $i<j<k<l$.
Then
\begin{eqnarray*}
    \psi (P_{i,l})\psi (P_{j,k})
    &=&(\eta (c_{l-2})\cdots \eta (c_i))_{11} \psi (P_{j,k})\\
  &=&\sum _{r=1}^2(\eta (c_{l-2})\cdots \eta (c_k))_{1r}
  (\eta (c_{k-1})\cdots \eta (c_i))_{r1}\psi (P_{j,k})\\
  &=& (\psi (P_{k,l})\psi (P_{i,k+1})-\psi (P_{k+1,l})\psi (P_{i,k}))
  \psi (P_{j,k}),
\end{eqnarray*}
where the last equation follows from Lemma~\ref{le:eta}.
Apply Equation~\eqref{eq:psiprop2} to the first summand of the last expression.
Then
\begin{eqnarray*}
    \psi (P_{i,l})\psi (P_{j,k})
    &=& \psi (P_{k,l})\psi (P_{i,k})\psi (P_{j,k+1})-
    \psi (P_{k,l})\psi (P_{i,j})\\
    && -\psi (P_{k+1,l})\psi (P_{i,k})\psi (P_{j,k})\\
    &=&\psi (P_{i,k})\psi (P_{j,l})-\psi (P_{k,l})\psi (P_{i,j}),
\end{eqnarray*}
where the second equation is obtained from the definition
of $\psi (P_{j,l})$
and Lemma~\ref{le:eta}. The obtained formula proves that $\psi $ is
well-defined. Since $\psi (P_{i,i+1})=1$ for all $i\in \{1,2,\dots,n-1\}$,
it induces an algebra map $\bar{\psi }:\cG (2,n)\to R/J$. Further,
Lemma~\ref{le:RJgen}
tells that $R/J$ is
generated by the elements $c_1,\dots,c_{n-2}$. Since
$\bar{\psi }\bar{\varphi }(c_i)=\bar{\psi }(P_{i,i+2})=c_i$
for all $i\in \{1,2,\dots,n-2\}$, we conclude that
$\bar{\psi }\bar{\varphi }=\id _{R/J}$, and hence $\bar{\varphi }$ is
injective. Clearly, $\bar{\psi }$ is the inverse of $\bar{\varphi }$.
\end{proof}

\begin{remar}
Alternatively, we can prove that $J=\ker \varphi$ by constructing a
matrix $z=(z_{ij})_{i,j}\in (R/J)^{2\times n}$ such that if
\[ \detz(i,j) := \det
\begin{pmatrix} z_{1i} & z_{1j} \\ z_{2i} & z_{2j}\end{pmatrix}
\quad \mbox{for} \quad 1\le i,j\le n, \]
then $\detz(i,i+1)=1$ and $\detz(i,i+2)=c_i$ for all $i=1,\ldots,n$,
where all indices are to be read modulo $n$ and
$\detz(n-1,1)=-c_{n-1}$, $\detz(n,2)=-c_n$, $\detz(n,1)=-1$,
in the following way:

Let $\xi^{(i)} = -\eta(c_i)^{-1}\cdots\eta(c_2)^{-1}\eta(c_1)^{-1}$
and write $\bar\xi^{(i)} = (\xi^{(i)})^{-1}$; this is well-defined
over $R/J$ since $\det \eta(x) = 1$.
For $i\equiv 1 \:(\mbox{mod } 4)$ set
\begin{align*}
z_{1i} &=& \xi^{(i-1)}_{11}, \quad
z_{1\,i+1} &=& \xi^{(i)}_{11}, \quad
z_{1\,i+2} &=& -\bar\xi^{(i)}_{21}, \quad
z_{1\,i+3} &=& -\bar\xi^{(i+1)}_{21}, \\
z_{2i} &=& \xi^{(i-1)}_{12}, \quad
z_{2\,i+1} &=& \xi^{(i)}_{12}, \quad
z_{2\,i+2} &=& \bar\xi^{(i)}_{11}, \quad
z_{2\,i+3} &=& \bar\xi^{(i+1)}_{11}.
\end{align*}
(Ignore the cases for which $i$, $i+1$, $i+2$ or $i+3$ are greater than $n$.)
Several explicit calculations give that
$\detz(i,i+1)=1$ for $i=1,\ldots,n-1$
and $\detz(i,i+2)=c_i$ for $i=1,\ldots,n-2$
hold even in $R$. For
$\detz(n-1,1)=-c_{n-1}$, $\detz(n,2)=-c_n$, and $\detz(n,1)=-1$
we use Equation~\eqref{eta_id}.

Recall that the determinants satisfy the Pl\"ucker relations. Hence, once the
matrix $z$ is constructed, it follows that there is an algebra
map from $\cG (2,n)$ to $R/J$ which maps $P_{i,i+2}$ to $c_i$.
\end{remar}

\begin{remar}
Equation~\eqref{equ1} gives a rule how to produce reflection
groupoids of rank two with explicit Cartan entries: Start with a triangle and
inductively append triangles. If a new triangle is glued to an edge
with label $y$, then the new labels $a,b$ will satisfy $ab=y$.
If all edges of the $n$-gon have value $1$, then one has completed
a reflection groupoid.
\end{remar}

\bibliographystyle{amsplain}
\bibliography{rg_r2_caA}

\providecommand{\bysame}{\leavevmode\hbox to3em{\hrulefill}\thinspace}
\providecommand{\MR}{\relax\ifhmode\unskip\space\fi MR }
\providecommand{\MRhref}[2]{%
  \href{http://www.ams.org/mathscinet-getitem?mr=#1}{#2}
}
\providecommand{\href}[2]{#2}
\begin{thebibliography}{10}

\bibitem{b-BourLie4-6}
N.~Bourbaki, \emph{Groupes et alg{\`e}bres de {L}ie, ch.\,4, 5 et 6},
  {\'E}l{\'e}ments de math{\'e}matique, Hermann, Paris, 1968.

\bibitem{p-CH09-3}
M.~Cuntz and I.~Heckenberger, \emph{Finite {W}eyl groupoids of rank three}, in
  preparation (2009).

\bibitem{p-CH09-2}
\bysame, \emph{Weyl groupoids of rank two and continued fractions}, Algebra \&
  Number Theory \textbf{3} (2009), 317--340.

\bibitem{p-CH09}
\bysame, \emph{Weyl groupoids with at most three objects}, J. Pure Appl.
  Algebra \textbf{213} (2009), no.~6, 1112--1128.

\bibitem{FR07}
S.~Fomin and N.~Reading, \emph{Root systems and generalized associahedra},
  Geometric combinatorics, IAS/Park City Math. Ser., vol.~13, Amer. Math. Soc.,
  Providence, RI, 2007, pp.~63--131.

\bibitem{b-GP-00}
M.~Geck and G.~Pfeiffer, \emph{Characters of finite {C}oxeter groups and
  {I}wahori-{H}ecke algebras}, London Mathematical Society Monographs. New
  Series, vol.~21, The Clarendon Press Oxford University Press, New York, 2000.

\bibitem{p-G-09}
B.~Gr{\"u}nbaum, \emph{A catalogue of simplicial arrangements in the real
  projective plane}, Ars Math. Contemp. \textbf{2} (2009), no.~1, 1--25.

\bibitem{HW}
G.~H. Hardy and E.~M. Wright, \emph{An introduction to the theory of numbers},
  fourth ed., The Clarendon Press Oxford University Press, 1975.

\bibitem{p-H06}
I.~Heckenberger, \emph{The {W}eyl groupoid of a {N}ichols algebra of diagonal
  type}, Invent. Math. \textbf{164} (2006), 175--188.

\bibitem{p-HY-08}
I.~Heckenberger and H.~Yamane, \emph{A generalization of {C}oxeter groups, root
  systems, and {M}atsumoto's theorem}, Math. Z. \textbf{259} (2008), no.~2,
  255--276.

\bibitem{b-Kac90}
Victor~G. Kac, \emph{Infinite-dimensional {L}ie algebras}, third ed., Cambridge
  University Press, Cambridge, 1990.

\end{thebibliography}

\end{document}